\documentclass[english]{article}
\usepackage[T1]{fontenc}
\usepackage[latin1]{inputenc}
\usepackage{geometry}
\usepackage{float}
\usepackage{mathrsfs}
\usepackage{amsmath}
\usepackage{amssymb}
\usepackage{graphicx}
\usepackage{subfigure}
\usepackage[textsize=tiny]{todonotes}

\makeatletter

\floatstyle{ruled}
\newfloat{algorithm}{tbp}{loa}
\providecommand{\algorithmname}{Algorithm}
\floatname{algorithm}{\protect\algorithmname}
\newcommand{\red}[1]{{\color{red} #1}}

\usepackage{amsthm}

\usepackage{mathrsfs}

\usepackage{amsfonts}

\usepackage{epsfig}

\usepackage{mathrsfs}

\usepackage{enumerate}

\newtheorem{theorem}{Theorem}[section]
\newtheorem{lem}{Lemma}[section]
\newtheorem{rem}{Remark}[section]
\newtheorem{prop}{Proposition}[section]

\newcounter{hypA}
\newenvironment{hypA}{\refstepcounter{hypA}\begin{itemize}
  \item[({\bf A\arabic{hypA}})]}{\end{itemize}}

\newcounter{hypB}
\newenvironment{hypB}{\refstepcounter{hypB}\begin{itemize}
  \item[({\bf B\arabic{hypB}})]}{\end{itemize}}

\newcommand{\bbE}{\mathbb{E}}

\usepackage{xspace}
\usepackage{tabu}
\usepackage{booktabs}

\usepackage{babel}\date{}

\usepackage{babel}

\makeatother

\begin{document}

\begin{center}

{\Large \textbf{Markov chain Simulation for Multilevel Monte Carlo}}

\vspace{0.5cm}

BY AJAY JASRA$^{1}$, KODY LAW$^{2}$, \& YAXIAN XU$^{3}$ 

{\footnotesize $^{1,3}$Department of Statistics \& Applied Probability,
National University of Singapore, Singapore, 117546, SG.}
{\footnotesize E-Mail:\,}\texttt{\emph{\footnotesize staja@nus.edu.sg}, \emph{\footnotesize a0078115@u.nus.edu}}\\
{\footnotesize $^{2}$School of Mathematics,
University of Manchester, Manchester, M13 9PL, UK.}
{\footnotesize E-Mail:\,}\texttt{\emph{\footnotesize kodylaw@gmail.com}}
\end{center}

\begin{abstract}
This paper considers a new approach to using Markov chain Monte Carlo (MCMC) in contexts where one may adopt
multilevel (ML) Monte Carlo. The underlying problem is to approximate expectations w.r.t.~an
underlying probability measure that is associated to a continuum problem, such as a continuous-time stochastic process.
It is then assumed that the associated probability measure can
only be used (e.g.~sampled) under a discretized approximation. In such scenarios, it is known that to achieve a target error, the computational
effort can be reduced when using MLMC relative to exact sampling from the most accurate discretized probability. 
The ideas rely upon introducing hierarchies of the discretizations where
less accurate approximations cost less
to compute, 
and using an appropriate collapsing 
sum expression for the target expectation. 
If a suitable coupling of the 
probability measures in the hierarchy is achieved, then 
a reduction in cost is possible.
This article focused on the case where 
exact sampling from such coupling is not possible. 
We show that one can construct suitably coupled MCMC kernels when given only access to MCMC kernels
which are invariant with respect to each discretized probability measure.
We prove, under assumptions,
that this coupled MCMC approach in a ML context 
can reduce the cost to achieve a given error, relative to exact sampling. 
Our approach is illustrated on a numerical example.
\end{abstract}

\section{Introduction}

Consider a probability measure $\pi$ on a measurable space $(\mathsf{X},\mathcal{X})$. 
For a collection of $\pi-$integrable and measurable functions $\varphi:\mathsf{X}\rightarrow\mathbb{R}$,
we are interested in computing expectations:
$$
\pi(\varphi) := \int_{\mathsf{X}} \varphi(x)\pi(dx).
$$
It is assumed that the exact value of $\pi(\varphi)$ is not available analytically and must be approximated numerically: one such approach is the Monte Carlo method, 
which is focused upon in this article. 

There is an additional complication in the context of this paper; we assume that the probability measure is not available, even for simulation, and must be approximated. More precisely, we
assume that $\pi$ is associated to a continuum problem, such as Bayesian inverse problems
(e.g.~\cite{hoan:14,stuart}), where one cannot evaluate the target exactly. 
As an example, $\pi$ may be associated
to the solution of a partial differential equation (PDE) which needs to be approximated numerically.
Such contexts arise in a wide variety of real applications; see \cite{law2015data, stuart, oliver} and the references therein.
Thus, we assume that there exists a discretized approximation of $\pi$ say $\pi_L$ on $(\mathsf{X},\mathcal{X})$, where $L$ is a potentially vector-valued parameter which controls the quality of the approximation
and also associates to the cost of computations w.r.t.~$\pi_L$. In other words, suppose $L\in\mathbb{N}_0$ and that for $\varphi:\mathsf{X}\rightarrow\mathbb{R}$ $\pi,\pi_L-$integrable
$$
\lim_{L\rightarrow\infty}|\pi_L(\varphi)-\pi(\varphi)| = 0
$$
where $\pi_L(\varphi) = \int_{\mathsf{X}} \varphi(x)\pi_L(dx)$ and the cost associated to computing $\pi_L$ also grows (without bound) with $L$. Examples of practical models
with such properties can be found in \cite{beskos,hoan:12}.

In the context outlined above, if exact sampling from $\pi_L$ is possible, then one solution to approximating $\pi_L(\varphi)$ is the use of the Monte Carlo
method, by sampling i.i.d.~from $\pi_L$ and using the approximation $\frac{1}{N}\sum_{i=1}^{N}\varphi(X_i)$, where $X_i\stackrel{\textrm{i.i.d.}}{\sim}\pi_L$.
It is well known that such a method can be improved upon using multilevel \cite{giles,giles1,hein} or multi-index \cite{mimc} Monte Carlo (MIMC) methods; we focus upon the former in this introduction, but note
that our subsequent remarks broadly apply to the latter. The basic notion of the MLMC method is to introduce a collapsing sum representation
$$
\pi_L(\varphi) = \pi_0(\varphi) + \sum_{l=1}^L [\pi_l-\pi_{l-1}](\varphi)
$$
where  $\pi_0,\dots,\pi_L$ is a hierarchy of probability measures on $(\mathsf{X},\mathcal{X})$, which are approximations of $\pi$, 
starting with the very inaccurate and computationally cheap $\pi_0$ and up to the most precise 
and most computationally expensive $\pi_L$
and $\varphi:\mathsf{X}\rightarrow\mathbb{R}$ is assumed to be integrable w.r.t.~each of the probability measures. The idea is then, if one can sample $\pi_0$ exactly and a sequence
of (dependent) couplings of $(\pi_{l-1},\pi_l)$ exactly, then it is possible to reduce the cost, relative to i.i.d.~sampling from $\pi_L$ to achieve a prespecified mean square error (MSE). That is, writing $\mathbb{E}$ as an
expectation w.r.t.~the algorithm which approximates $\pi(\varphi)$, with an estimate $\widehat{\pi}(\varphi)$, the MSE is $\mathbb{E}[(\widehat{\pi}(\varphi)-\pi(\varphi))^2]$. The key to the cost reduction is sampling
from a coupling of $(\pi_{l-1},\pi_l)$ which is `good enough'; see \cite{giles,giles1} for details.

In this article, we focus on the scenario where one can only hope to sample $\pi_0,\dots,\pi_L$ using some Markov chain method such as MCMC. In such scenarios it can be non-trivial to sample 
from good couplings of $(\pi_{l-1},\pi_l)$. There has been substantial work on this, including MCMC \cite{scheichlmlmcmc2013,hoan:12,jasra,ourmimcmc} and sequential Monte Carlo (SMC) \cite{beskos,beskos1,delm_mlnc,mlpf}; see \cite{ml_rev}  for a review of these ideas.
In the context of interest, most of the methods used \cite{beskos,beskos1,delm_mlnc,hoan:12,jasra,ourmimcmc} rely upon replacing, to some extent, coupling with importance sampling (the only exception to our knowledge is 
\cite{scheichlmlmcmc2013}) and then performing the relevant sampling from some appropriate sequence of change of measures using MCMC or SMC. These procedures often require some `good' change of measures, just as
good couplings are required in MLMC. 
In the MIMCMC case, no exact proof of the improvement brought about by using MIMCMC is given (\cite{ourmimcmc} only provide a proof for a simplified identity). 

The main motivation
of the methodology we introduce is to provide an approach which only requires a simple (and reasonable) MCMC algorithm to sample $\pi_0,\dots,\pi_L$. 
We focus on approximating $[\pi_l-\pi_{l-1}](\varphi)$ independently for each $1\leq l \leq L$.
Our approach does not seem to be used in the ML or MI
literature: representing the Markov chain transitions invariant w.r.t.~$\pi_{l-1}$ and $\pi_l$ as an iterated map, one can simply couple the simulation of simple random variables that are used commonly in the sampling of the
iterated map for $\pi_{l-1}$ and $\pi_l$ respectively (this is defined explicitly in Section \ref{sec:methods}). It is straightforward
to establish that such an approach can provide consistent estimates of $[\pi_l-\pi_{l-1}](\varphi)$. We also remark that many well known Markov transitions such as Metropolis-Hastings or deterministic scan Gibbs samplers can
be represented as an iterated map.
The main issue is to establish that such a coupled approximation can be useful, in the sense that there
is a reduction in cost, relative to MCMC from $\pi_L$, to achieve a prespecified MSE. We show that under appropriate assumptions, that this can indeed be the case. 
The approach discussed in this paper is generally best used in the case where the MCMC kernel is \emph{rejection free}, such as a Gibbs sampler or some non-reversible MCMC algorithms \cite{bouncy}. The basic idea outlined here can also be easily extended to the
context where MIMC can be beneficial, but both the implementation and mathematical analysis are left to future work.

This article is structured as follows. In Section \ref{sec:methods} we describe our method. In Section \ref{sec:theory} we give our mathematical results. In Section \ref{sec:numerics}
we provide a numerical examples illustrating our method. 
A summary is provided in Section \ref{sec:summary}.
Mathematical results are given in appendices \ref{app:var_res} and \ref{app:verify}. 

\section{Methodology}\label{sec:methods}

\subsection{Notations}

Let $(\mathsf{X},\mathcal{X})$ be a measurable space. 
For $\varphi:\mathsf{X}\rightarrow\mathbb{R}$ we write $\mathscr{B}_b(\mathsf{X})$
and $\textrm{Lip}(\mathsf{X})$ as the collection of bounded measurable and Lipschitz functions respectively.
For $\varphi\in\mathscr{B}_b(\mathsf{X})$, we write the supremum norm $\|\varphi\|=\sup_{x\in\mathsf{X}}|\varphi(x)|$.
For $\varphi\in\mathscr{B}_b(\mathsf{X})$, $\textrm{Osc}(\varphi)=\sup_{(x,y)\in\mathsf{X}\times\mathsf{X}}|\varphi(x)-\varphi(y)|$.
For $\varphi\in\textrm{Lip}(\mathsf{X})$, we write the Lipschitz constant 
$\|\varphi\|_{\textrm{Lip}}:=\sup_{(x,y)\in\mathsf{X}\times\mathsf{X}}\frac{|\varphi(x)-\varphi(y)|}{|x-y|}$, 
where $|\cdot|$ is used to denote the $\mathbb{L}_1-$norm.
$\mathscr{P}(\mathsf{X})$ (resp.~$\mathscr{M}(\mathcal{X})$) denotes the collection of probability measures (resp.~$\sigma-$finite measures) on $(\mathsf{X},\mathcal{X})$.
For a measure $\mu$ on $(\mathsf{X},\mathcal{X})$
and a $\varphi\in\mathscr{B}_b(\mathsf{X})$, the notation $\mu(\varphi)=\int_{\mathsf{X}}\varphi(x)\mu(dx)$ is used.
Let $K:\mathsf{X}\times\mathcal{X}\rightarrow[0,1]$ be a Markov kernel (e.g.~\cite{meyn} for a definition) and $\mu\in \mathscr{M}(\mathcal{X})$ 
and $\varphi\in\mathscr{B}_b(\mathsf{X})$, then 
we use the notations
$
\mu K(dy) = \int_{\mathsf{X}}\mu(dx) K(x,dy)
$
and for $\varphi\in\mathscr{B}_b(\mathsf{X})$, 
$
K(\varphi)(x) = \int_{\mathsf{X}} \varphi(y) K(x,dy).
$
For a Markov kernel $K$ we write the $n-$iterates
$$
K^n(x_0,dx_n) = \int_{\mathsf{X}} K^{n-1}(x_0,dx_{n-1})K(x_{n-1},dx_n) 
$$
with $K^0(x,dy)=\delta_x(dy)$.
For $\mu,\nu\in\mathscr{P}(\mathsf{X})$, the total variation distance 
is written $\|\mu-\nu\|_{\textrm{tv}}=\sup_{A\in\mathcal{X}}|\mu(A)-\nu(A)|$.
For $A\in\mathcal{X}$ the indicator is written $\mathbb{I}_A(x)$. For a sequence
$(a_n)_{n\geq 0}$  and for $0\leq k \leq n$, we use the compact notation 
$a_{k:n}=(a_k,\dots,a_n)$, with the convention that if $k>n$ the resulting vector of objects is null.

\subsection{Set-Up}

We begin by concentrating upon a sequence $(\pi_l)_{0\leq l \leq L}$, with $\pi_l\in \mathscr{P}(\mathsf{X})~\forall l\in\{1,\dots,L\}$, where $L$ can potentially increase, but is taken
as fixed.
It is assumed that there is a limiting $\pi\in \mathscr{P}(\mathsf{X})$, in the sense that for any $\varphi\in\mathscr{B}_b(\mathsf{X})$
$$
\lim_{l\rightarrow+\infty}|\pi_l(\varphi)-\pi(\varphi)| = 0.
$$
As discussed in the introduction, our objective is to approximate the ML identity:
\begin{equation}\label{eq:ml_id}
\pi_L(\varphi) = \pi_0(\varphi) + \sum_{l=1}^L [\pi_l-\pi_{l-1}](\varphi).
\end{equation}
We assume that each $\pi_l$ is associated to a scalar parameter $h_l$, 
with $1>h_0>\cdots>h_L>0$, which represents the quality of the approximation w.r.t.~$\pi$
and the associated cost.
Approximation of \eqref{eq:ml_id} can be peformed via Monte Carlo intergration, by sampling from dependent couplings of the pairs $(\pi_l,\pi_{l-1})$, independently for $1\leq l \leq L$,
and i.i.d.~sampling from $\pi_0$. We are focussed upon the scenario where exact sampling even from a given $\pi_l$ is not currently possible, but can be achieved
by sampling a $\pi_l-$invariant Markov kernel, as in MCMC.

\subsection{Approach}

\subsubsection{Multilevel Approach}

Let $0\leq l \leq L$ be given. Let $(\mathsf{U},\mathcal{U})$ be a finite-dimensional measuable space, $U$ a random variable
on $(\mathsf{U},\mathcal{U})$ with probability $\mu\in\mathscr{P}(\mathsf{U})$ and let $\xi_l:\mathsf{X}\times\mathsf{U}\rightarrow\mathsf{X}$,
such that the map $(x,u)\rightarrow\xi_{l}(x,u)$ is jointly measurable on the product $\sigma-$algebra $\mathcal{X}\vee \mathcal{U}$. Consider
the discrete-time Markov chain, for $X_0\in\mathsf{X}$ given, $n\geq 1$:
$$
X_n = \xi_l(X_{n-1},U_n)
$$
where $(U_n)_{n\geq 1}$ is a sequence of i.i.d.~random variables with probability $\mu$. Denote the associated Markov kernel $K_l(x,dy)$. We
explicitly assume that $\xi_l$ is constructed such that $K_l$ admits $\pi_l$ as an invariant measure; examples are given in Section \ref{sec:ex_iterated_map}.
Under appropriate assumptions on $\xi_l$ (e.g.~\cite{diaconis,jarner}) or $K_l$ (e.g.~\cite{meyn}), one can show that
$\|K^n(x,\dot)-\pi_l(\cdot)\|_{\textrm{tv}}$ will converge to zero as $n$ grows. In addition, under appropriate assumptions, 
for $\varphi\in\mathscr{B}_b(\mathsf{X})$,
$$
\frac{1}{N}\sum_{n=1}^N \varphi(X_n)
$$
will converge almost surely to $\pi_l(\varphi)$.

As noted, we are interested in the approximation of $[\pi_l-\pi_{l-1}](\varphi)$, $1\leq l \leq L$ as the approximation of $\pi_0(\varphi)$ is possible using the above
discussion. The approach that is proposed in this article is as follows. Let $1\leq l \leq L$ be given and set $\check{X}_0(l)=(\overline{X}_0(l),\underline{X}_0(l))=(X_0(l),X_0(l))\in\mathsf{X}\times\mathsf{X}$. Generate the Markov chain for $n\geq 1$
$$
\overline{X}_n(l) = \xi_l(\overline{X}_{n-1}(l),U_n(l)) \quad \textrm{and}\quad \underline{X}_n(l) = \xi_{l-1}(\underline{X}_{n-1}(l),U_n(l))
$$
with the notation $\check{X}_n(l)=(\overline{X}_n(l),\underline{X}_n(l))$. Note that the sequence of random numbers, $(U_n(l))_{n\geq 1}$, are the same
for the recursion of $\overline{X}_n(l)$ and $\underline{X}_n(l)$. The Markov kernel associated to the Markov chain $(\check{X}_n(l))_{n\geq 0}$
is denoted $\check{K}_{l,l-1}$.
Then if this scheme is repeated independently for each $1\leq l \leq L$ and one samples
the Markov chain $(X_n(0))_{n\geq 1}$ using the Markov kernel $K_0$ we have the following estimate of \eqref{eq:ml_id}
$$
\frac{1}{N_0}\sum_{n=1}^{N_0} \varphi(X_n(0))
+ \sum_{l=1}^L \frac{1}{N_l}\sum_{n=1}^{N_l}[\varphi(\overline{X}_n(l))-\varphi(\underline{X}_n(l))].
$$
Under appropriate assumptions on $\xi_l$ or $K_l$, one can easily prove that the above estimate is consistent. The question is whether this
can improve upon sampling from $\pi_L$ using $K_L$ (in the sense discussed in the introduction), in some scenarios of practical interest. 
This point is considered in Section \ref{sec:theory}.

A strategy for coupling MCMC is considered in \cite{sergios}, except when trying to perform unbiased estimation using the approach in \cite{rg:15}. 
That idea is different to the one 
presented in this article, and it is not analyzed in the context of variances, as is the case in this paper.

\subsubsection{Examples of $\xi_l$}\label{sec:ex_iterated_map}

There are numerous examples of mappings which fall into the framework of this article. 
Let us suppose that for any $0\leq l \leq L$ we have
$$
\pi_l(dx) = \pi_l(x) \nu(dx)
$$
with $\nu\in\mathscr{M}(\mathcal{X})$ and $\pi(x)$ a non-negative function that is known up-to a constant.

A wide class of Metropolis-Hastings kernels are one such example, which is taken from the description in \cite{chen}.
Let $x_n$ be the current state of the Markov chain. A proposal $Y=\psi_l(x_n,\red{U}_1)$ is generated. 
This may be for instance a (one-dimensional) Gaussian random walk,
where 
$\psi_l(x_n,U_1):=x_n + U_1$, $U_1\sim\mathscr{N}(0,\Sigma_l)$, 
and $\mathscr{N}(0,\Sigma_l)$ is the Gaussian distribution of mean zero and variance $\Sigma_l$.
Assuming that the proposal can be written
$$
Q_l(x,dy) = q_l(x,y)\nu(dy)
$$
for some non-negative function $q_l$ that is known, then we have the well-defined acceptance probability
$$
a_l(x,y) = \left\{\begin{array}{ll}  
1\wedge \frac{\pi_l(y)q_l(y,x)}{\pi_l(x)q_l(x,y)} & \textrm{if}~(x,y)\in R_l\\
0 & \textrm{otherwise}
\end{array}\right.
$$
where $R_l=\{(x,y):\pi_l(y)q_l(y,x)>0~\textrm{and}~\pi_l(x)q_l(x,y)>0\}$. Then, if $U_2\sim\mathscr{U}_{[0,1]}$
where $\mathscr{U}_{[0,1]}$ is the uniform distribution on $[0,1]$, (so here $U=U_{1:2}$)
$$
\xi_l(x_n,u_{1:2}) = \left\{\begin{array}{ll} 
\psi_l(x_n,u_1) & \textrm{if}~u_2< a_l(x_n,\psi_l(x_n,u_1)) \\
x_n & \textrm{otherwise}.
\end{array}\right.
$$
\cite{chen} also demonstrate that the deterministic scan Gibbs sampler falls into our framework (in fact one can 
establish that the random scan can also do so).

\begin{rem}
The approach detailed here assumes that the $\pi_l\in \mathscr{P}(\mathsf{X})~\forall l\in\{0,\dots,L\}$. However, this need
not be the case. For instance it is straightforward to extend to the case where $\pi_l\in \mathscr{P}(\mathsf{X}_l)$ and
$\mathsf{X}_0\subset \mathsf{X}_1 \subset \cdots \subset \mathsf{X}_L$. Here one would simply use the same random numbers
in the mappings $\xi_l$, on the common parts of the space. For example, if $\mathsf{X}_l = \mathbb{R}^l$, then in the random walk
example above, one simply uses the same Gaussian perturbation in the proposal on the space $\mathbb{R}^{l-1}$ and one additional
Gaussian must be sampled for the level $l$. The same uniform is used in the accept/reject step.
We note that the theory in Section \ref{sec:theory}
could be extended to the scenario we mentioned, with similar, but more complicated, arguments.
\end{rem}

\section{Theoretical Results}\label{sec:theory}

\subsection{Assumptions}

Throughout $\mathsf{X}$ is compact. Recall that $1>h_0>\cdots>h_L>0$.


\begin{hypA}\label{ass:1}
There exist $C<\infty$ and $\rho\in(0,1)$ such that for any $n\geq 1$
$$
\sup_{l\geq 0} \sup_{x\in \mathsf{X}} \|K_l^n(x,\dot)-\pi_l(\cdot)\|_{\textrm{tv}} \leq C \rho^n.
$$
\end{hypA}
\begin{hypA}\label{ass:3}
There exist a $C<\infty$ such that for any $l\geq 1$, $\varphi\in\mathscr{B}_b(\mathsf{X})\cap\textrm{Lip}(\mathsf{X})$, $(x,y)\in\mathsf{X}\times\mathsf{X}$
$$
|K_l(\varphi)(x)-K_l(\varphi)(y)| \leq C (\|\varphi\|\vee \|\varphi\|_{\textrm{Lip}})|x-y|.
$$
\end{hypA}
\begin{hypA}\label{ass:2}
There exist a $C<\infty$ and $\beta >0$ such that for any $l\geq 1$, $\varphi\in\mathscr{B}_b(\mathsf{X})$
\begin{enumerate}
\item{
$
|[\pi_l-\pi_{l-1}](\varphi)| \leq C \|\varphi\| h_l^{\beta}.
$
}
\item{
$
\sup_{x\in\mathsf{X}}|[K_l(\varphi)(x)-K_{l-1}(\varphi)(x)| \leq C  \|\varphi\| h_l^{\beta}.
$
}
\end{enumerate}
\end{hypA}
\begin{hypA}\label{ass:4}
There exist a $\tau\in(0,1)$ such that for any $l\geq 1$, $(x,y)\in\mathsf{X}\times\mathsf{X}$
$$
\int_{\mathsf{U}}|\xi_l(x,u)-\xi_l(y,u)|^2\mu(du) \leq \tau|x-y|^2.
$$
\end{hypA}
\begin{hypA}\label{ass:5}
There exist a $C<\infty$ and $\beta >0$ such that for any $l\geq 1$, $x\in\mathsf{X}$
$$
\int_{\mathsf{U}}|\xi_l(x,u)-\xi_{l-1}(x,u)|^2\mu(du) \leq C h_l^{\beta}.
$$
\end{hypA}

These assumptions are quite strong, but may be verified in practice. As $\mathsf{X}$ is compact (A\ref{ass:1}) can be verified for a wide class of
Markov kernels, see for instance \cite{meyn}. (A\ref{ass:3}-\ref{ass:2}) relate to the continuity properties of the kernel and the underlying ML problem.
In particular (A\ref{ass:2}) 2.~states that moves under pairs of kernels at consecutive levels, stay close on average, given that they are initialised at the same point.
(A\ref{ass:4}) is an assumption that has been considered by \cite{jarner} in a slightly more general context and relates to contractive properties of the iterated map.
(A\ref{ass:5}) is a map analogue of (A\ref{ass:2}) 2.. One expects that such properties are needed, in order for the coupling approach here to work well.
The compactness and uniform ergodicity assumptions could be weakened to say geometric or polynomial type ergodicity and non-compact spaces with longer,
but similar, arguments (see e.g.~\cite{andrieu}).

In general for \emph{rejection} type Markov chains, such as Metropolis-Hastings kernels, (A\ref{ass:1}), (A\ref{ass:2}) 2.~and (A\ref{ass:5}) may be in conflict.
(A\ref{ass:2}) 2.~and (A\ref{ass:5}) suggest that the moves of the kernels at consecutive levels as well as the coupling are sufficiently close and good
respectively. However, (A\ref{ass:1}) demands a reasonably fast convergence rate that is independent of the level. This latter assumption can be made
$l-$dependent, but the key point is that the mixing rate should not go to zero with $h_l$. Conversely, for (A\ref{ass:2}) 2.~and (A\ref{ass:5}), one
needs that if one level accepts and the other rejects, that the resulting positions of the chains are sufficiently close, as a function of $h_l$. For instance
for (A\ref{ass:2}) 2.~with Gaussian random walk proposals with different scalings at consecutive levels, one will typically need a scale of $\mathcal{O}(h_l^{\beta})$,
which will likely mean a reduction in the rate of convergence and possibly negate the advantage of a multilevel method. This is rather unsurprising, in that one wishes to keep consecutive levels close via coupling.

The main message is then: for rejection free algorithms such as the Gibbs sampler or some non-reversible MCMC kernels one should simply check if the couplings are `good enough' in the sense
of (A\ref{ass:5}) (one suspects that (A\ref{ass:4}) could be significantly weakened). Then one expects ML improvements. For MCMC methods with rejection, the same must be done, but,
it will not be clear (without mathematical calculations or numerical trials) that this method of implementing the coupling of MCMC will yield reductions in computational effort for a given level of MSE.

\subsection{Main Result}

The main idea is to consider 
$$
\mathbb{E}\Big[\Big(
\frac{1}{N_0}\sum_{n=1}^{N_0}[\varphi(X_n(0)) - \pi_0(\varphi)] + 
\sum_{l=1}^L\frac{1}{N_l}\sum_{n=1}^{N_l}\{[\varphi(\overline{X}_n(l))-\varphi(\underline{X}_n(l))]-[\pi_l-\pi_{l-1}](\varphi)\}
-
$$
$$
[\pi_L-\pi](\varphi)
\Big)^2\Big].
$$
This error is clearly equal to
$$
\mathbb{E}\Big[\Big(\frac{1}{N_0}\sum_{n=1}^{N_0}[\varphi(X_n(0)) - \pi_0(\varphi)]\Big)^2\Big] 
+ \sum_{l=1}^L \mathbb{E}\Big[\Big(\frac{1}{N_l}\sum_{n=1}^{N_l}\{[\varphi(\overline{X}_n(l))-\varphi(\underline{X}_n(l))]-[\pi_l-\pi_{l-1}](\varphi)\}\Big)^2\Big]
+
$$
$$
\sum_{l\neq q =1}^L \bbE [ f_l f_q ]
+ [\pi_L-\pi](\varphi)^2 \, ,
$$
where 
\begin{equation}\label{eq:bias_mcmc}
f_0 := \frac{1}{N_0}\sum_{n=1}^{N_l}[\varphi(X_n(0)) - \pi_0(\varphi)] \, ,
\end{equation}
and for $j=1,\dots, L$,
$$
f_j := \frac{1}{N_j}\sum_{n=1}^{N_j}\{[\varphi(\overline{X}_n(j))-\varphi(\underline{X}_n(j))]-[\pi_j-\pi_{j-1}](\varphi)\} \, .
$$
The terms \eqref{eq:bias_mcmc} and 
$$
\mathbb{E}\Big[\Big(\frac{1}{N_0}\sum_{n=1}^{N_0}[\varphi(X_n(0)) - \pi_0(\varphi)]\Big)^2\Big]
$$
can be treated by standard Markov chain theory for $\varphi$ in an appropriate class and under our assumptions.
That is, under (A\ref{ass:1}), one can easily prove that there exist a $C<+\infty$ such that for any $N_0\geq 1$, $\varphi\in\mathscr{B}_b(\mathsf{X})\cap\textrm{Lip}(\mathsf{X})$
$$
\Big|\mathbb{E}\Big[\Big(\frac{1}{N_0}\sum_{n=1}^{N_l}[\varphi(X_n(0)) - \pi_0(\varphi)]\Big)\Big]\Big| \leq \frac{C  (\|\varphi\|\vee \|\varphi\|_{\textrm{\textrm{Lip}}}) }{N_0}
$$
and 
$$
\mathbb{E}\Big[\Big(\frac{1}{N_0}\sum_{n=1}^{N_0}[\varphi(X_n(0)) - \pi_0(\varphi)]\Big)^2\Big] \leq \frac{C  (\|\varphi\|\vee \|\varphi\|_{\textrm{\textrm{Lip}}})^2 }{N_0}
$$
For the remaining terms, we have the following results, whose proofs can be found in Appendix \ref{app:var_res}.

\begin{prop}\label{prop:marg}
Assume (A\ref{ass:1},\ref{ass:2}). Then there exist a $C<+\infty$ such that for any $l\geq 1$, $N_l\geq 1$, $\varphi\in\mathscr{B}_b(\mathsf{X})\cap\textrm{\emph{Lip}}(\mathsf{X})$:
$$
\Big|\mathbb{E}\Big[\Big(\frac{1}{N_l}\sum_{n=1}^{N_l}\{[\varphi(\overline{X}_n(l))-\varphi(\underline{X}_n(l))]-[\pi_l-\pi_{l-1}](\varphi)\}\Big)\Big]\Big| \leq \frac{C  (\|\varphi\|\vee \|\varphi\|_{\textrm{\textrm{\emph{Lip}}}}) h_l^{\beta}}{N_l}.
$$
\end{prop}

\begin{theorem}\label{theo:main_thm}
Assume (A\ref{ass:1}-\ref{ass:5}). Then there exist a $C<+\infty$ such that for any $l\geq 1$, $N_l\geq 1$, $\varphi\in\mathscr{B}_b(\mathsf{X})\cap\textrm{\emph{Lip}}(\mathsf{X})$:
$$
\mathbb{E}\Big[\Big(\frac{1}{N_l}\sum_{n=1}^{N_l}\{[\varphi(\overline{X}_n(l))-\varphi(\underline{X}_n(l))]-[\pi_l-\pi_{l-1}](\varphi)\}\Big)^2\Big] \leq \frac{C  (\|\varphi\|\vee \|\varphi\|_{\textrm{\textrm{\emph{Lip}}}})^2 h_l^{\beta}}{N_l}.
$$
\end{theorem}

\begin{rem}
The case where $\varphi$ depends upon $l$, i.e.~$\varphi_l:\mathsf{X}\rightarrow\mathbb{R}$ could also be treated using the approach in Appendix \ref{app:var_res}. This would require some
additional calculations, but we believe a similar result to Theorem \ref{theo:main_thm} would hold on some minor additional assumptions on $\varphi_l$.
\end{rem}

\subsection{Verifying the Assumptions}

We consider a deterministic scan Gibbs sampler. We set $\mathsf{X}=\bigotimes_{i=1}^k \mathsf{X}_i$, $\mathcal{X}=\bigvee_{i=1}^k \mathcal{X}_i$ and remark that $\mathsf{X}_i$ need
not be a one-dimensional space (but must be compact). We set for $0\leq l \leq L$
$$
\pi_l(dx_{1:k}) = \pi_l(x_{1:k})\nu(dx_{1:k})
$$
where in an abuse of notation, we use $\pi_l$ to denote the density and measure, and $\nu(dx_{1:k})=\bigotimes_{i=1}^k\nu_i(dx_i)$, $\nu_i\in\mathscr{M}(\mathsf{X}_i)$ is the dominating measure, typically Lebesgue or counting.
Then for $0\leq l \leq L$ we set
\begin{equation}\label{eq:dgs_ver}
K_{l}(x_{1:k},dx_{1:k}') = \Big(\prod_{i=1}^k \pi_{l}(x_i'|x_{1:i-1}',x_{i+1:k})\Big)\nu(dx_{1:k}')
\end{equation}
where for $i\in\{1,\dots,k\}$
$$
\pi_{l}(x_i'|x_{1:i-1}',x_{i+1:k}) = \frac{\pi_{l}(x_{1:i}',x_{i+1:k})}{\int_{\mathsf{X}_i}\pi_{l}(x_{1:i}',x_{i+1:k})\nu_i(dx_i')}.
$$
We further suppose that $\pi_{l}(x_i|x_{1:i-1}',x_{i+1:k})$ can be sampled by $U_i\sim\mu_i$, $\mu_i\in\mathscr{P}(\mathsf{U}_i)$
($\mathsf{U}=\bigotimes_{i=1}^k \mathsf{U}_i$)
and using the transformation
\begin{equation}\label{eq:dgs_map_ver}
T_{l,i}([x_{1:i-1}',x_{i+1:k}],U_i)
\end{equation}
for each $l,i$.

\begin{hypB}\label{ass:verify}
\begin{enumerate}
\item{There exists $0< \underline{C}<\overline{C}<+\infty$ such that for any $l\geq 0$, $x_{1:k}\in\mathsf{X}$
$$
\underline{C} \leq \pi_l(x_{1:k}) \leq \overline{C}.
$$
}
\item{There exists $C<+\infty$ such that for any $l\geq 0$, $(x_{1:k},y_{1:k})\in\mathsf{X}\times\mathsf{X}$
$$
|\pi_l(x_{1:k})-\pi_l(y_{1:k})| \leq C |x_{1:k}-y_{1:k}|.
$$
}
\item{There exists $C<+\infty$ such that for any $l\geq 1$, $x_{1:k}\in\mathsf{X}$
$$
|\pi_l(x_{1:k})-\pi_{l-1}(x_{1:k})| \leq C h_l^{\beta}.
$$
}
\item{There exists $\tau<1/k$ such that for any $l\geq 0$, $(x_{1:k},y_{1:k})\in\mathsf{X}\times\mathsf{X}$, $1\leq i \leq k$, $u_i\in\mathsf{U}_i$
$$
|T_{l,i}([x_{1:k}],u_i)-T_{l,i}([y_{1:k}],u_i)|^2 \leq \tau|x_{1:k}-y_{1:k}|^2.
$$
}
\item{There exists $C<+\infty$ such that for any $l\geq 1$, $x_{1:k}\in\mathsf{X}$$, 1\leq i \leq k$, $u_i\in\mathsf{U}_i$
$$
|T_{l,i}([x_{1:k}],u_i)-T_{l-1,i}([x_{1:k}],u_i)|^2 \leq C h_l^{\beta}.
$$
}
\end{enumerate}
\end{hypB}

\begin{prop}\label{prop:verify}
Assume (B\ref{ass:verify}). Then the target $\pi_l$ and deterministic Gibbs sampler corresponding to \eqref{eq:dgs_ver} which is sampled via  \eqref{eq:dgs_map_ver}, satisfies 
(A\ref{ass:1}-\ref{ass:5}).
\end{prop}
\noindent The proof is given in Appendix \ref{app:verify}.

\section{Numerical Experiment}\label{sec:numerics}

\subsection{Model and MCMC}

We consider a slightly modified model in \cite{sergios1}. Data $y_1,\dots$, $y_i\in\mathbb{R}$ are such
that for $u_i\in\mathbb{R}$, $\lambda\in\mathbb{R}^+$
$$
Y_i|u_i \stackrel{ind}{\sim} \mathscr{N}(u_i,\lambda^{-1}).
$$
$\lambda\in\mathbb{R}^+$ is assumed to be known.
The prior model on the $u_i$ are associated to an unknown hyper-parameter $\delta\in\mathbb{R}^+$ and we assume a joint 
prior for any $k\geq 1$
$$
p(u_{1:k},\delta) = \exp\Big\{-\frac{\delta}{2}\sum_{j=1}^k j^{-3}u_j^2\Big\}\delta^{\alpha_0-1}\exp\{-\delta\kappa_0\} \, ,
$$
which is improper if $k\geq 2\alpha_0$.
It is easily checked that, for any $k\geq 1$, the posterior on $(u_{1:k},\delta)$ is proper provided 
$(\alpha_0,\kappa_0,\lambda)\in(\mathbb{R}^+)^3$.

Set $h_l=K_l^{-1}$, where $K_l=M_0 2^l$; 
for some $M_0\in\mathbb{N}$, we consider a sequence of posteriors for $l\in\{0,1,\dots,L\}$ on $u_{1:K_l},\delta$ (the data are simulated from the true model). One can construct a Gibbs sampler as in \cite{sergios1} 
which has full conditional densities:
\begin{eqnarray*}
u_{1:K_l} | y_{1:K_l}, \delta & \sim & \mathscr{N}_{K_l}(m_l(\delta),C_l(\delta)) \\
\delta | y_{1:K_l}, u_{1:K_l} & \sim & \mathscr{G}(\alpha_0,\kappa_l) 
\end{eqnarray*}
where  $\mathscr{N}_{h}(m,C)$ is a $h-$dimensional Gaussian distribution of mean vector $m$ and covariance matrix $C$, 
$m_l(\delta)= \lambda^{-1}\Big(\frac{y_1}{\delta+\lambda^{-1}},\dots,\frac{ y_{K_l }}
{\delta(h_{l})^{3}+\lambda^{-1}}\Big)'$,
$C_l(\delta)=\textrm{diag}((\delta+\lambda^{-1})^{-1},\dots,(\delta(h_{l})^{3}+\lambda^{-1})^{-1})$, $\mathscr{G}(\alpha,\kappa)$ is a Gamma
distribution of mean $\alpha/\kappa$ ($(\alpha,\kappa)\in(\mathbb{R}^+)^2$) and $\kappa_l = \kappa_0 + \frac{1}{2}\sum_{i=1}^{K_l}i^{-3} u_i^2$.

The Gibbs sampler is easily coupled when considering levels $l$ and $l-1$, $l\geq 1$.
Denote 
by $L_l(\delta)$ the lower triangular Cholesky factor
of $C_l(\delta)$.
 When considering the update at level $l$, $\check{u}(l)_{1:K_l} | y_{1:K_l}, \delta$,
one samples a $V_{1:K_l}\sim\mathscr{N}_{K_l}(0,I_{K_l})$ ($I_{K_l}$ is the $K_l\times K_l$ identity matrix) and sets
$$
\overline{U}(l)_{1:K_l} = m_l(\delta) + 
L_l(\delta) V_{1:K_l}
$$
and 
$$
\underline{U}(l)_{1:K_{l-1}} = m_{l-1}(\delta) +L_{l-1}(\delta)
V_{1:K_{l-1}} \, .
$$
Note that the value of $\delta$ will typically be different between levels $l$ and $l-1$. 
For the update on $\delta$, one can sample $W\sim\mathscr{G}(\alpha_0,1)$ and
set 
$\overline{\delta}(l) = \kappa_l W$ and 
$\underline{\delta}(l)=\kappa_{l-1} W$.

\begin{rem}
The paper \cite{sergios1} focuses on improving the algorithm described here by a reformulation
of the problem which requires an accept/reject step. We do not consider the improved sampling algorithm, 
as the objective here is to illustrate that the method developed in this paper works for rejection-free MCMC. 
\end{rem}

\subsection{Simulation Results}

In our experiments $\alpha_0=1$, $\beta_0=0.1$, $\lambda=1000$ and $M_0=8$. 
We consider the posterior expectation of the function $\frac{1}{M_0}\sum_{i=1}^{M_0} u_i$. 
This expectation appears to converge to a limit as $l$ grows. 
We first run an experiment to determine the $\beta$ (as in Theorem \ref{theo:main_thm}) and order of bias of
the approximation (expected to be $\mathcal{O}(h_l^{\rho})$ for some $\rho\geq \beta/2$). The true value is computed by using the algorithm at one plus the most precise level (i.e.~$L+1$).
The results are presented in Figure \ref{fig:fig1}, which suggest that $\rho=2$ and $\beta=4$. The cost of the algorithm at level $l$ is $\mathcal{O}(h_l^{-1})$. Using the standard ML theory (e.g.~\cite{giles}), for $\epsilon>0$ given we set $N_l=\mathcal{O}(\epsilon^{-2} h_l^{2.5})$ and $L=\mathcal{O}(|\log(\epsilon)|)$.
In Figure \ref{fig:fig2} we can see the cost against MSE for the MLMCMC procedure, versus the MCMC at level $L$ with $N=\mathcal{O}(\epsilon^{-2})$ samples. The improvement
is clear.

\begin{figure}\centering
\centering \subfigure[Variance]{{\includegraphics[width=0.49\textwidth,height=8cm]{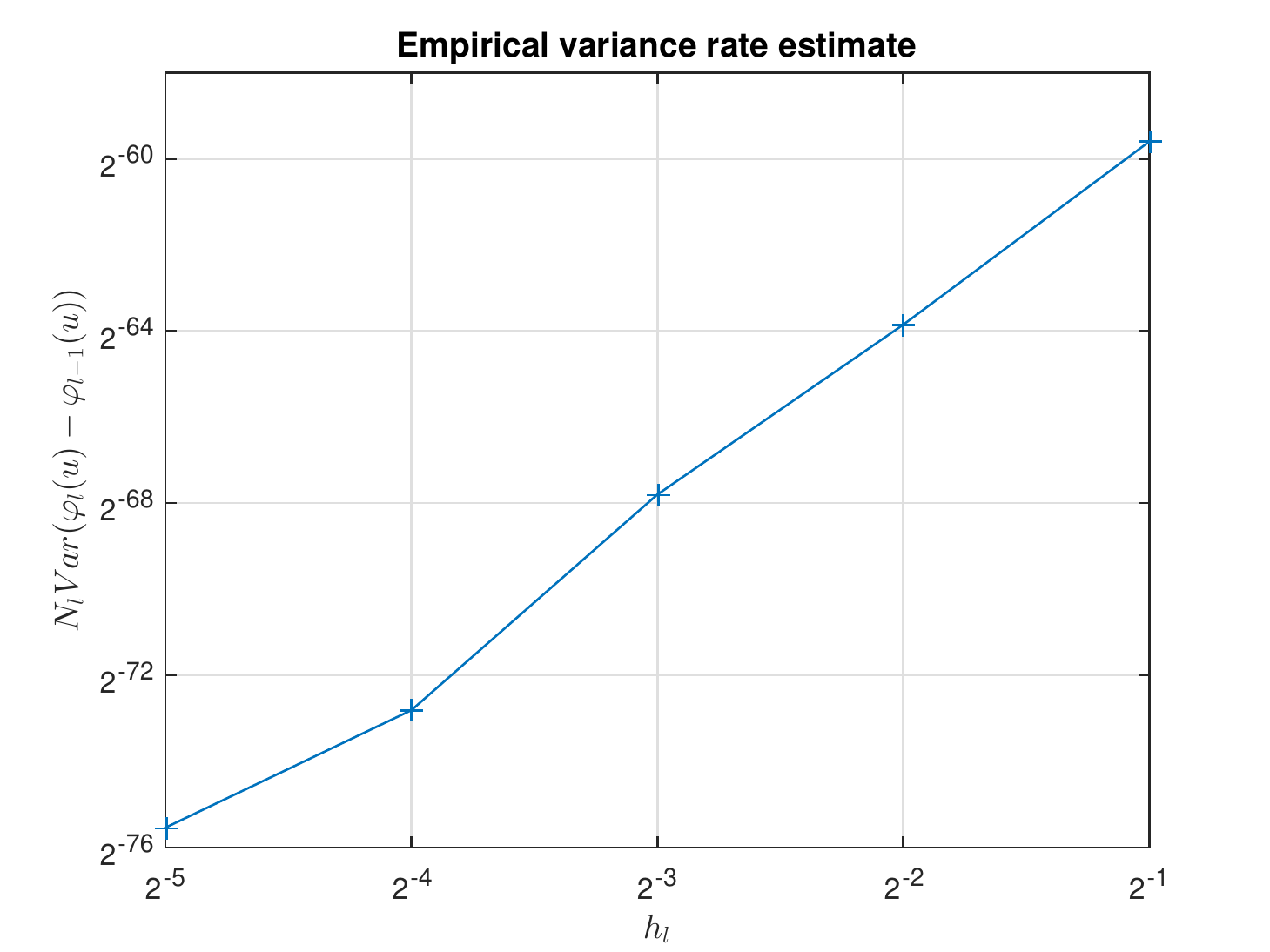}}}
\subfigure[Bias]{{\includegraphics[width=0.49\textwidth,height=8cm]{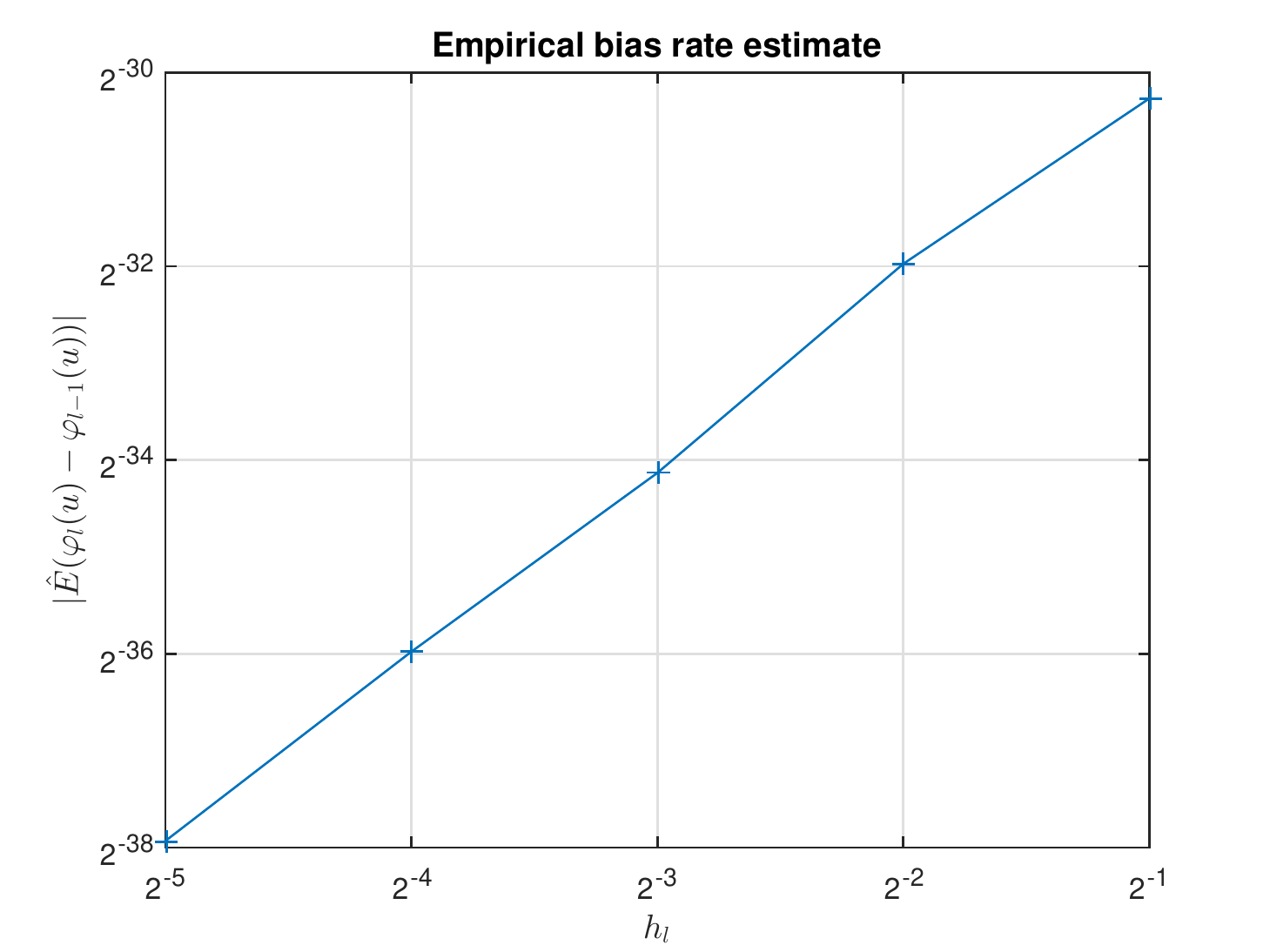}}}
\caption{Estimates of the Variance and Bias Rates.}
\label{fig:fig1}
\end{figure}

\begin{figure}\centering
  \includegraphics[height=8cm]{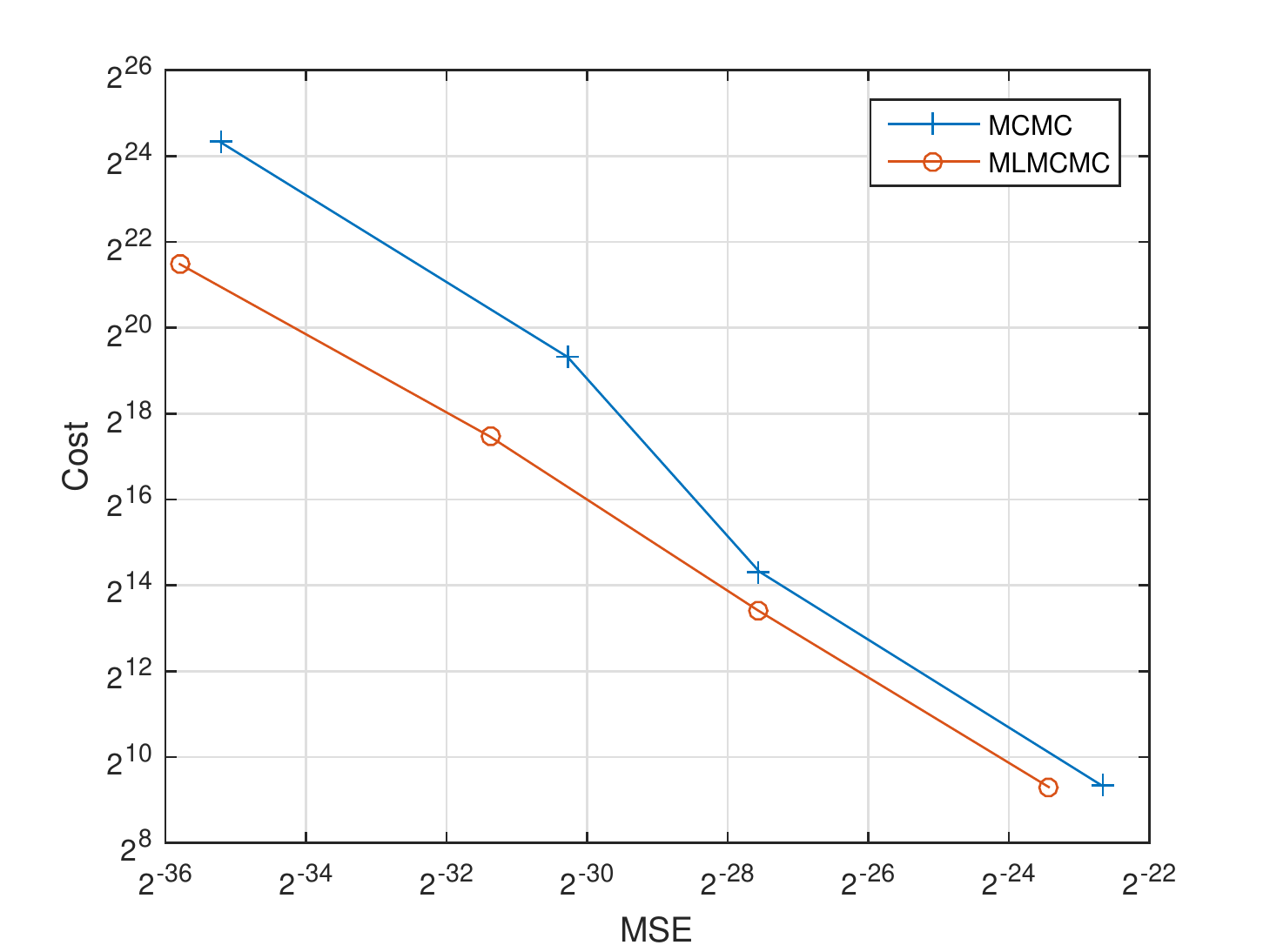}
  \caption{Cost against MSE.}
  \label{fig:fig2}
\end{figure}

\section{Summary}\label{sec:summary}

In this article we have considered a new Markov chain simulation approach to implementing MLMC. The main utility to using this
idea, is that one needs only a standard MCMC procedure for sampling $\pi_0,\dots,\pi_L$. Then the additional implementation is straightforward
and can sometimes drastically reduce the cost to achieve a given level of MSE. As we have remarked, one does require some properties of the Markov
kernel simulated, in order for the strategy that we have suggested to work well in practice. In general, we believe it is of most use if the Markov chain method
is rejection free. In contexts where this cannot be achieved, we believe a more intricate coupling is required 
(see e.g.~\cite{bou,heng}).

\subsubsection*{Acknowledgements}
AJ was supported by an AcRF tier 2 grant: R-155-000-161-112. AJ is affiliated with the Risk Management Institute, the Center for Quantitative Finance 
and the OR \& Analytics cluster at NUS. AJ was supported by a KAUST CRG4 grant ref: 2584.
KJHL was supported by the University of Manchester School of Mathematics. 

\appendix

\section{Variance Results}\label{app:var_res}

In order to prove Proposition \ref{prop:marg} and Theorem \ref{theo:main_thm} we will make use of the (additive) Poisson equation (e.g.~\cite{glynn}).
That is, let $\varphi\in\mathscr{B}_b(\mathsf{X})$ then given our Markov kernels $(K_l)_{l\geq 1}$ we say that $\widehat{\varphi}_l$ solves
the Poisson equation if
$$
\varphi(x) - \pi_l(\varphi) = \widehat{\varphi}_l(x) - K_l(\widehat{\varphi}_l)(x).
$$
If (A\ref{ass:1}) holds, then a solution is
$$
\widehat{\varphi}_l(x) = \sum_{n\geq 0} [K_l^n(\varphi)(x)-\pi_l(\varphi)]
$$
which is the one that is considered here.

Our proof consists of several lemmata, followed by the proofs of Proposition \ref{prop:marg} and Theorem \ref{theo:main_thm}. Throughout $C$ is a finite constant which may change from line-to-line,
but will not depend on any important parameters (e.g.~$l$, $n$) unless stated.

\begin{lem}\label{lem:pois_lip}
Assume (A\ref{ass:1}-\ref{ass:3}). Then there exists a $C<\infty$, such that for any $l\geq 0$, $\varphi\in\mathscr{B}_b(\mathsf{X})\cap\textrm{\emph{Lip}}(\mathsf{X})$
we have:
$$
|\widehat{\varphi}_l(x)-\widehat{\varphi}_l(y)| \leq C(\|\varphi\|\vee \|\varphi\|_{\textrm{\emph{Lip}}}) |x-y|.
$$
\end{lem}

\begin{proof}
We have
$$
\widehat{\varphi}_l(x)-\widehat{\varphi}_l(y) = \varphi(x) - \varphi(y) + K_l(\varphi)(x)  - K_l(\varphi)(y).
$$
The proof follows then by the triangular inequality and (A\ref{ass:3}).
\end{proof}

\begin{lem}\label{lem:poisson_eq_cont}
Assume (A\ref{ass:1},\ref{ass:2}). Then there exists a $C<\infty$, such that for any $l\geq 1$, $\varphi\in\mathscr{B}_b(\mathsf{X})$ we have:
$$
\sup_{x\in\mathsf{X}} |\widehat{\varphi}_l(x)-\widehat{\varphi}_{l-1}(x)| \leq C\|\varphi\| h_l^{\beta}.
$$
\end{lem}

\begin{proof}
By \cite[Proposition C.2]{andrieu} we have
$$
\widehat{\varphi}_l(x)-\widehat{\varphi}_{l-1}(x) = 
$$
$$
\sum_{n\geq 1}\Big[
\sum_{i=0}^{n-1} [K_l^i-\pi_l]\big\{[K_l-K_{l-1}]\big([K_{l-1}^{n-i-1}-\pi_{l-1}](\varphi)\big)\big\}(x) -
[\pi_l-\pi_{l-1}]\{[K_{l-1}^n-\pi_{l-1}](\varphi)\}
\Big].
$$
We consider the two terms in the difference of the summand individually.

\noindent\textbf{Term:} $\sum_{i=0}^{n-1} [K_l^i-\pi_l]\big\{[K_l-K_{l-1}]\big([K_{l-1}^{n-i-1}-\pi_{l-1}](\varphi)\big)\big\}(x)$.
Now
$$
[K_l-K_{l-1}]\big([K_{l-1}^{n-i-1}-\pi_{l-1}](\varphi)\big)(x)
=
$$
$$
\textrm{Osc}(\varphi) \|K_{l-1}^{n-i-1}-\pi_{l-1}\|_{\textrm{tv}} \Big[[K_l-K_{l-1}]\Big(
\frac{[K_{l-1}^{n-i-1}-\pi_{l-1}](\frac{\varphi}{\textrm{Osc}(\varphi)})}{\|K_{l-1}^{n-i-1}-\pi_{l-1}\|_{\textrm{tv}}}
\Big)(x)\Big].
$$
By (A\ref{ass:1}) $\|K_{l-1}^{n-i-1}-\pi_{l-1}\|_{\textrm{tv}}\leq C\rho^{n-i-1}$ and 
$$
\frac{[K_{l-1}^{n-i-1}-\pi_{l-1}](\frac{\varphi}{\textrm{Osc}(\varphi)})}{\|K_{l-1}^{n-i-1}-\pi_{l-1}\|_{\textrm{tv}}}
$$
is bounded, so we have by (A\ref{ass:2})
\begin{equation}\label{eq:prf1}
\|[K_l-K_{l-1}]\big([K_{l-1}^{n-i-1}-\pi_{l-1}](\varphi)\big)\| \leq C \|\varphi\| \rho^{n-i-1} h_l^{\beta}
\end{equation}
for some $C$ that does not depend upon $n,i,l$. Now
$$
[K_l^i-\pi_l]\big\{[K_l-K_{l-1}]\big([K_{l-1}^{n-i-1}-\pi_{l-1}](\varphi)\big)\big\}(x) = 
$$
$$
\textrm{Osc}\Big([K_l-K_{l-1}]\big([K_{l-1}^{n-i-1}-\pi_{l-1}](\varphi)\big)\Big)
[K_l^i-\pi_l]\Big\{\frac{[K_l-K_{l-1}]\big([K_{l-1}^{n-i-1}-\pi_{l-1}](\varphi)\big)}{\textrm{Osc}\Big([K_l-K_{l-1}]\big([K_{l-1}^{n-i-1}-\pi_{l-1}](\varphi)\big)\Big)}\Big\} \leq
$$
$$
2\|[K_l-K_{l-1}]\big([K_{l-1}^{n-i-1}-\pi_{l-1}](\varphi)\big)\| \|K_l^i-\pi_l\|_{\textrm{tv}}
$$
so by (A\ref{ass:1}) and \eqref{eq:prf1}
$$
\|[K_l^i-\pi_l]\big\{[K_l-K_{l-1}]\big([K_{l-1}^{n-i-1}-\pi_{l-1}](\varphi)\big)\big\}\| \leq  C \|\varphi\| \rho^{n-1} h_l^{\beta}.
$$
Hence
\begin{equation}\label{eq:prf2}
\sum_{i=0}^{n-1} [K_l^i-\pi_l]\big\{[K_l-K_{l-1}]\big([K_{l-1}^{n-i-1}-\pi_{l-1}](\varphi)\big)\big\} \leq C n\|\varphi\|\rho^{n-1} h_l^{\beta}.
\end{equation}

\noindent\textbf{Term:} $[\pi_l-\pi_{l-1}]\{[K_{l-1}^n-\pi_{l-1}](\varphi)\}$.
Clearly
$$
[\pi_l-\pi_{l-1}]\{[K_{l-1}^n-\pi_{l-1}](\varphi)\} = 
\textrm{Osc}(\varphi)\Big\|[K_{l-1}^n-\pi_{l-1}]\Big(\frac{\varphi}{\textrm{Osc}(\varphi)}\Big)\Big\|
[\pi_l-\pi_{l-1}]\Big\{\frac{[K_{l-1}^n-\pi_{l-1}](\varphi)}{\|[K_{l-1}^n-\pi_{l-1}](\varphi)\|}\Big\}.
$$
Then, via (A\ref{ass:1}), as
$$
\Big\|[K_{l-1}^n-\pi_{l-1}]\Big(\frac{\varphi}{\textrm{Osc}(\varphi)}\Big)\Big\| \leq \sup_{x\in\mathsf{X}} \|K_{l-1}^n(s,\cdot)-\pi_{l-1}(\cdot)\|_{\textrm{tv}} \leq C \rho^{n}
$$
it follows by (A\ref{ass:2}) that 
\begin{equation}\label{eq:prf3}
|[\pi_l-\pi_{l-1}]\{[K_{l-1}^n-\pi_{l-1}](\varphi)\}| \leq C \|\varphi\|\rho^{n-1} h_l^{\beta}.
\end{equation}

The proof is completed by combining \eqref{eq:prf2} and \eqref{eq:prf3} and noting that the associated upper-bounds are summable in $n$.
\end{proof}

\begin{lem}\label{lem:sq_osc}
Assume (A\ref{ass:4}-\ref{ass:5}). Then there exists a $C<\infty$, such that for any $l\geq 1$, $n\geq 1$, $\varphi\in\mathscr{B}_b(\mathsf{X})\cap\textrm{\emph{Lip}}(\mathsf{X})$
we have:
$$
\mathbb{E}[(\varphi(\overline{X}_n(l))-\varphi(\underline{X}_n(l)))^2] \leq C(\|\varphi\|\vee \|\varphi\|_{\textrm{\emph{Lip}}})^2 h_l^{\beta}.
$$
\end{lem}

\begin{proof}
For notational simplicity, we drop the argument $l$ from the random variables $(\overline{X}_n(l),\underline{X}_n(l)$ $,U_n(l))$, i.e.~we just write
$(\overline{X}_n,\underline{X}_n,U_n)$. For $l\geq 0$ we denote the iterated map $\xi_l^n(X_0,U_{1:n}) = \xi_l(\xi_l^{n-1}(X_0,U_{1:n-1}),U_n)$,
with $n\geq 1$ and the convention that $\xi_l^{0}(X_0,U_{-1})=X_0$. Then we have
\begin{equation}\label{eq:h_prf4}
\mathbb{E}[(\varphi(\overline{X}_n)-\varphi(\underline{X}_n))^2] = \mathbb{E}[(\varphi(\xi_l^n(x_0,U_{1:n}))-\varphi(\xi_{l-1}^n(x_0,U_{1:n})))^2].
\end{equation}
Then one has the decomoposition
$$
\mathbb{E}[(\varphi(\xi_l^n(x_0,U_{1:n}))-\varphi(\xi_{l-1}^n(x_0,U_{1:n})))^2] =
$$
$$
\mathbb{E}\Big[\Big(\sum_{k=1}^n\{
\varphi(\xi_l^{n-k+1}(\xi_{l-1}^{k-1}(x_0,U_{1:k-1}),U_{k:n})) - 
\varphi(\xi_l^{n-k}(\xi_{l-1}^{k}(x_0,U_{1:k}),U_{k+1:n}))
\}\Big)^2\Big].
$$
Now, applying Minkowski and using the fact that $\varphi\in\textrm{Lip}(\mathsf{X})$
$$
\mathbb{E}[(\varphi(\xi_l^n(x_0,U_{1:n}))-\varphi(\xi_{l-1}^n(x_0,U_{1:n})))^2] \leq
$$
\begin{equation}\label{eq:h_prf1}
\|\varphi\|_{\textrm{Lip}}^2\Big(\sum_{k=1}^n\mathbb{E}\Big[\Big|\xi_l^{n-k+1}(\xi_{l-1}^{k-1}(x_0,U_{1:k-1}),U_{k:n})-\xi_l^{n-k}(\xi_{l-1}^{k}(x_0,U_{1:k}),U_{k+1:n})\Big|^2\Big]^{1/2}\Big)^2.
\end{equation}
Let $\mathcal{U}_n=\sigma(U_{1:n})$ (i.e.~the $\sigma-$algebra generated by $U_{1:n}$)
and consider the summand:
$$
\mathbb{E}\Big[\Big(\xi_l^{n-k+1}|\xi_{l-1}^{k-1}(x_0,U_{1:k-1}),U_{k:n})-\xi_l^{n-k}(\xi_{l-1}^{k}(x_0,U_{1:k}),U_{k+1:n})\Big|^2\Big] = 
$$
$$
\mathbb{E}\Big[\mathbb{E}[\int_{\mathsf{U}}|\xi_l(\xi_l^{n-k}(\xi_{l-1}^{k-1}(x_0,u_{1:k-1}),u_{k:n-1}),u_n)-\xi_l(\xi_l^{n-k-1}(\xi_{l-1}^{k}(x_0,u_{1:k}),u_{k+1:n-1}),u_n)|^2\times
$$
$$
\mu(du_n)|\mathcal{U}_{n-1}]\Big].
$$
Applying (A\ref{ass:4}) gives the upper-bound
$$
\mathbb{E}\Big[\Big|\xi_l^{n-k+1}|\xi_{l-1}^{k-1}(x_0,U_{1:k-1}),U_{k:n})-\xi_l^{n-k}(\xi_{l-1}^{k}(x_0,U_{1:k}),U_{k+1:n})\Big|^2\Big] \leq
$$
$$
\tau\mathbb{E}\Big[|\xi_l^{n-k}(\xi_{l-1}^{k-1}(x_0,U_{1:k-1}),U_{k:n-1})-\xi_l^{n-k-1}(\xi_{l-1}^{k}(x_0,U_{1:k}),U_{k+1:n-1})\Big|^2\Big].
$$
Thus, applying this argument recursively, yields
$$
\mathbb{E}\Big[\Big|\xi_l^{n-k+1}(\xi_{l-1}^{k-1}(x_0,U_{1:k-1}),U_{k:n})-\xi_l^{n-k}(\xi_{l-1}^{k}(x_0,U_{1:k}),U_{k+1:n})\Big|^2\Big] \leq
$$
\begin{equation}\label{eq:h_prf2}
\tau^{n-k}\mathbb{E}\Big[\Big|\xi_l(\xi_{l-1}^{k-1}(x_0,U_{1:k-1}),U_k) - \xi_{l-1}^{k}(x_0,U_{1:k})\Big|^2\Big].
\end{equation}
Then as
$$
\mathbb{E}\Big[\Big|\xi_l(\xi_{l-1}^{k-1}(x_0,U_{1:k-1}),U_k) - \xi_{l-1}^{k}(x_0,U_{1:k})\Big|^2\Big] =
$$
$$ 
\mathbb{E}\Big[\mathbb{E}\Big[\int_{\mathsf{U}}|\xi_l(\xi_{l-1}^{k-1}(x_0,u_{1:k-1}),u_k)-
\xi_{l-1}(\xi_{l-1}^{k-1}(x_0,u_{1:k-1}),u_k)
|^2\mu(du_k)\Big|\mathcal{U}_{k-1}\Big]\Big]
$$
applying (A\ref{ass:5}) yields
\begin{equation}\label{eq:h_prf3}
\mathbb{E}\Big[\Big|\xi_l(\xi_{l-1}^{k-1}(x_0,U_{1:k-1}),U_k) - \xi_{l-1}^{k}(x_0,U_{1:k})\Big|^2\Big] \leq C h_l^{\beta}.
\end{equation}
Combining \eqref{eq:h_prf3} with \eqref{eq:h_prf2} yields
$$
\mathbb{E}\Big[\Big|\xi_l^{n-k+1}|\xi_{l-1}^{k-1}(x_0,U_{1:k-1}),U_{k:n})-\xi_l^{n-k}(\xi_{l-1}^{k}(x_0,U_{1:k}),U_{k+1:n})\Big|^2\Big] \leq
C \tau^{n-k} h_l^{\beta}.
$$
Then returning to \eqref{eq:h_prf1}, we have shown that
\begin{eqnarray*}
\mathbb{E}[(\varphi(\xi_l^n(x_0,U_{1:n}))-\varphi(\xi_{l-1}^n(x_0,U_{1:n})))^2] & \leq & C\|\varphi\|_{\textrm{Lip}}^2(\sum_{k=1}^n\tau^{(n-k)/2})^2 h_l^{\beta}\\  & \leq & C\|\varphi\|_{\textrm{Lip}}^2h_l^{\beta}.
\end{eqnarray*}
Noting \eqref{eq:h_prf4} one can conclude the result.
\end{proof}

\begin{lem}\label{lem:final_tech_lemma}
Assume (A\ref{ass:1}-\ref{ass:5}). Then there exists a $C<\infty$, such that for any $l\geq 1$, $n\geq 1$, $\varphi\in\mathscr{B}_b(\mathsf{X})\cap\textrm{\emph{Lip}}(\mathsf{X})$
we have:
\begin{enumerate}
\item{
$
\mathbb{E}\Big[\Big(\widehat{\varphi}_l(\overline{X}_n(l))-\widehat{\varphi}_{l-1}(\underline{X}_n(l)) + K_l(\widehat{\varphi}_l)(\overline{X}_{n-1}(l)) - K_{l-1}(\widehat{\varphi}_{l-1})(\underline{X}_{n-1}(l))\Big)^2\Big] \leq C(\|\varphi\|\vee \|\varphi\|_{\textrm{\emph{Lip}}})^2 h_l^{\beta}.
$
}
\item{
$
\mathbb{E}[(K_l(\widehat{\varphi}_l)(\overline{X}_n(l))-K_{l-1}(\widehat{\varphi}_{l-1})(\underline{X}_n(l)))^2] \leq C(\|\varphi\|\vee \|\varphi\|_{\textrm{\emph{Lip}}})^2 h_l^{\beta}.
$
}
\end{enumerate}
\end{lem}

\begin{proof}
As in the proof of Lemma \ref{lem:sq_osc}, we drop the argument $l$ from the random variables $(\overline{X}_n(l),\underline{X}_n(l))$.
We start with 1.. We have the decomposition
$$
\widehat{\varphi}_l(\overline{X}_n)-\widehat{\varphi}_{l-1}(\underline{X}_n) + K_l(\widehat{\varphi}_l)(\overline{X}_{n-1}) - K_{l-1}(\widehat{\varphi}_{l-1})(\underline{X}_{n-1}) = 
$$
$$
\widehat{\varphi}_l(\overline{X}_n) - \widehat{\varphi}_{l-1}(\overline{X}_n) + \widehat{\varphi}_{l-1}(\overline{X}_n)  
- \widehat{\varphi}_{l-1}(\underline{X}_n)
+ K_l(\widehat{\varphi}_l-\widehat{\varphi}_{l-1})(\overline{X}_{n-1}) +
$$
$$
K_l(\widehat{\varphi}_{l-1})(\overline{X}_{n-1}) - 
K_l(\widehat{\varphi}_{l-1})(\underline{X}_{n-1}) + K_l(\widehat{\varphi}_{l-1})(\underline{X}_{n-1}) - K_{l-1}(\widehat{\varphi}_{l-1})(\underline{X}_{n-1}).
$$
Then applying the $C_2-$inequality we have
$$
\mathbb{E}\Big[\Big(\widehat{\varphi}_l(\overline{X}_n(l))-\widehat{\varphi}_{l-1}(\underline{X}_n(l)) + K_l(\widehat{\varphi}_l)(\overline{X}_{n-1}(l)) - K_{l-1}(\widehat{\varphi}_{l-1})(\underline{X}_{n-1}(l))\Big)^2\Big] \leq C \sum_{j=1}^5 T_j
$$
where
\begin{eqnarray*}
T_1 & = & \mathbb{E}[(\widehat{\varphi}_l(\overline{X}_n) - \widehat{\varphi}_{l-1}(\overline{X}_n))^2] \\
T_2 & = & \mathbb{E}[(\widehat{\varphi}_{l-1}(\overline{X}_n)  
- \widehat{\varphi}_{l-1}(\underline{X}_n))^2] \\
T_3 & = & \mathbb{E}[K_l(\widehat{\varphi}_l-\widehat{\varphi}_{l-1})(\overline{X}_{n-1})^2] \\
T_4 & = & \mathbb{E}[(K_l(\widehat{\varphi}_{l-1})(\overline{X}_{n-1}) - 
K_l(\widehat{\varphi}_{l-1})(\underline{X}_{n-1}))^2] \\
T_5 & = & \mathbb{E}[(K_l(\widehat{\varphi}_{l-1})(\underline{X}_{n-1}) - K_{l-1}(\widehat{\varphi}_{l-1})(\underline{X}_{n-1}))^2].
\end{eqnarray*}
We need now bound each of the above terms. For $T_1$, applying Lemma \ref{lem:poisson_eq_cont} we have
$$
T_1 \leq C \|\varphi\|^2 h_l^{2\beta}.
$$
For $T_2$, by Lemma \ref{lem:pois_lip}, $\widehat{\varphi}_{l-1}\in\textrm{Lip}(\mathsf{X})$ with Lipschitz constant that is independent of $l$.
Similarly, by (A\ref{ass:1}) 
$$
|\widehat{\varphi}_{l-1}(x)| \leq \textrm{Osc}(\varphi)\sum_{n\geq 0}\Big|[K_{l-1}^n-\pi_{l-1}](x)\Big(\frac{\varphi}{\textrm{Osc}(\varphi)}\Big)\Big| \leq C\|\varphi\|
$$
so $\widehat{\varphi}_{l-1}\in\mathscr{B}_b(\mathsf{X})$ and that the sup-norm does not depend upon $l$.
Hence, applying Lemma \ref{lem:sq_osc}
$$
T_2 \leq C(\|\varphi\|\vee \|\varphi\|_{\textrm{Lip}})^2 h_l^{\beta}.
$$
For $T_3$, applying Lemma \ref{lem:poisson_eq_cont} we have
$$
T_3 \leq C \|\varphi\|^2 h_l^{2\beta}.
$$
For $T_4$, when $n\geq 2$, applying (A\ref{ass:3}) along with the same argument in the proof of Lemma \ref{lem:sq_osc} yields
$$
T_4 \leq C(\|\varphi\|\vee \|\varphi\|_{\textrm{Lip}})^2 h_l^{\beta}.
$$
If $n=1$ the result also holds as $T_4=0$. Finally for $T_5$ by (A\ref{ass:3}) 2.~we have
$$
T_5 \leq C \|\varphi\|^2 h_l^{2\beta}.
$$
Putting together the above arguments verifies 1..

For 2..~this can be upper-bounded by a constant a term which is similar to $\sum_{j=3}^5 T_j$, except that the time sub-script in each of the $T_j$ is $n$, not $n-1$.
Hence a similar proof can be constructed as for 1..~and is hence omitted.
\end{proof}

\begin{proof}[Proof of Proposition \ref{prop:marg}]
We have
$$
\mathbb{E}\Big[\Big(\frac{1}{N_l}\sum_{n=1}^{N_l}[\varphi(\overline{X}_n(l))-\varphi(\underline{X}_n(l))]-[\pi_l-\pi_{l-1}](\varphi)\Big)\Big]
$$
$$
= \frac{1}{N}\sum_{n=1}^{N_l}\{K_l^n(\varphi)(x_0) - K_{l-1}^n(\varphi)(x_0) - [\pi_l-\pi_{l-1}](\varphi)\}.
$$
Then as in Lemma \ref{lem:poisson_eq_cont}, using a similar result to \cite[Proposition C.2]{andrieu}, it follows
that
$$
\mathbb{E}\Big[\Big(\frac{1}{N_l}\sum_{n=1}^{N_l}[\varphi(\overline{X}_n(l))-\varphi(\underline{X}_n(l))]-[\pi_l-\pi_{l-1}](\varphi)\Big)\Big] = 
$$
$$
\frac{1}{N_l}
\sum_{n=1}^{N_l}\Big[
\sum_{i=0}^{n-1} [K_l^i-\pi_l]\big\{[K_l-K_{l-1}]\big([K_{l-1}^{n-i-1}-\pi_{l-1}](\varphi)\big)\big\}(x_0) -
[\pi_l-\pi_{l-1}]\{[K_{l-1}^n-\pi_{l-1}](\varphi)\}
\Big]
$$
The result follows from very similar calculations to those in the proof of Lemma \ref{lem:poisson_eq_cont} and is hence omitted.
\end{proof}

\begin{proof}[Proof of Theorem \ref{theo:main_thm}]
As in the proof of Lemma \ref{lem:sq_osc}, we drop the argument $l$ from the random variables $(\overline{X}_n(l),\underline{X}_n(l))$.
We have
$$
\mathbb{E}\Big[\Big(\frac{1}{N_l}\sum_{n=1}^{N_l}[\varphi(\overline{X}_n)-\varphi(\underline{X}_n)]-[\pi_l-\pi_{l-1}](\varphi)\Big)^2\Big] = 
$$
$$
\mathbb{E}\Big[\Big(\frac{1}{N_l}\sum_{n=1}^{N_l}[
\widehat{\varphi}_l(\overline{X}_n)-K_l(\widehat{\varphi}_l)(\overline{X}_{n})
-
 \{\widehat{\varphi}_{l-1}(\underline{X}_n) - K_{l-1}(\widehat{\varphi}_{l-1})(\underline{X}_{n})\}
]\Big)^2\Big] = 
$$
$$
\mathbb{E}\Big[\Big(
\frac{1}{N_l}\sum_{n=1}^{N_l}[
\widehat{\varphi}_l(\overline{X}_n)-K_l(\widehat{\varphi}_l)(\overline{X}_{n-1})
-
 \{\widehat{\varphi}_{l-1}(\underline{X}_n) - K_{l-1}(\widehat{\varphi}_{l-1})(\underline{X}_{n-1})\}]
$$
$$
+ R_l(N_l) - R_{l-1}(N_l)]\Big)^2\Big]
$$
where
\begin{eqnarray*}
R_l(N_l) & = & \frac{1}{N_l}[K_l(\widehat{\varphi}_l)(\overline{X}_{N_l})-K_l(\widehat{\varphi}_l)(x_{0})] \\
R_{l-1}(N_l) & = & \frac{1}{N_l}[K_{l-1}(\widehat{\varphi}_{l-1})(\underline{X}_{N_l})-K_{l-1}(\widehat{\varphi}_{l-1})(x_{0})]
\end{eqnarray*}
Set
$$
M_{N_l} = \sum_{n=1}^{N_l}[
\widehat{\varphi}_l(\overline{X}_n)-K_l(\widehat{\varphi}_l)(\overline{X}_{n-1})
-
 \{\widehat{\varphi}_{l-1}(\underline{X}_n) - K_{l-1}(\widehat{\varphi}_{l-1})(\underline{X}_{n-1})\}].
$$
Then, applying the $C_2-$inequality one has
\begin{equation}\label{eq:main_theo1}
\mathbb{E}\Big[\Big(\frac{1}{N_l}\sum_{n=1}^{N_l}[\varphi(\overline{X}_n)-\varphi(\underline{X}_n)]-[\pi_l-\pi_{l-1}](\varphi)\Big)^2\Big] \leq C \Big(\frac{1}{N_l^2}\mathbb{E}[M_{N_l}^2] + 
\mathbb{E}[((R_l(N_l) - R_{l-1}(N_l))^2] \Big).
\end{equation}
We deal with the two terms on the R.H.S.~separately.

\noindent\textbf{Term:} $\frac{1}{N_l^2}\mathbb{E}[M_{N_l}^2]$.
Let, for $n\geq 0$
$$
D_n = \Big[\widehat{\varphi}_l(\overline{X}_n)-\widehat{\varphi}_{l-1}(\underline{X}_n) + K_l(\widehat{\varphi}_l)(\overline{X}_{n-1}) - K_{l-1}(\widehat{\varphi}_{l-1})(\underline{X}_{n-1})\Big]
$$
with the convention that $D_0 = 0$.
Denoting $\mathcal{F}_n=\sigma(\overline{X}_{0:n},\underline{X}_{0:n})$ we can easily verify that $(M_n,\mathcal{F}_n)$ is a Martingale. So applying the Burkholder-Gundy-Davis inequality
we have
$$
\frac{1}{N_l^2}\mathbb{E}[M_{N_l}^2] \leq \frac{1}{N_l^2}\mathbb{E}\Big[\sum_{n=1}^{N_l}\{D_n-D_{n-1}\}^2\Big].
$$
By the $C_2-$inequality
$$
\frac{1}{N_l^2}\mathbb{E}[M_{N_l}^2] \leq \frac{2}{N_l^2}\sum_{n=1}^{N_l}\{\mathbb{E}[D_n^2] + \mathbb{E}[D_{n-1}^2]\}.
$$
Then applying Lemma \ref{lem:final_tech_lemma} 1.~we have
\begin{equation}\label{eq:main_theo2}
\frac{1}{N_l^2}\mathbb{E}[M_{N_l}^2] \leq \frac{C  (\|\varphi\|\vee \|\varphi\|_{\textrm{\textrm{Lip}}})^2 h_l^{\beta}}{N_l}.
\end{equation}

\noindent\textbf{Term:} $\mathbb{E}[((R_l(N_l) - R_{l-1}(N_l))^2]$
We have 
$$
\mathbb{E}[((R_l(N_l) - R_{l-1}(N_l))^2]  = 
$$
$$
\frac{1}{N_l^2}\mathbb{E}\Big[\Big(K_l(\widehat{\varphi}_l)(\overline{X}_{N_l})-K_{l-1}(\widehat{\varphi}_{l-1})(\underline{X}_{N_l}) +
K_{l-1}(\widehat{\varphi}_{l-1})(x_{0}) - K_l(\widehat{\varphi}_l)(x_{0})
\Big)^2\Big] \leq
$$
$$
\frac{1}{N_l^2}\Big(
\mathbb{E}\Big[\Big(K_l(\widehat{\varphi}_l)(\overline{X}_{N_l})-K_{l-1}(\widehat{\varphi}_{l-1})(\underline{X}_{N_l})\Big)^2\Big]
+
(K_{l-1}(\widehat{\varphi}_{l-1})(x_{0}) - K_l(\widehat{\varphi}_l)(x_{0}))^2
\Big). 
$$
Then applying Lemma \ref{lem:final_tech_lemma} 2.~we have
$$
\mathbb{E}[((R_l(N_l) - R_{l-1}(N_l))^2] \leq 
$$
$$
\frac{1}{N_l^2}\Big(C(\|\varphi\|\vee \|\varphi\|_{\textrm{Lip}})^2 h_l^{\beta} +
(K_{l-1}(\widehat{\varphi}_{l-1}-\widehat{\varphi}_l))(x_{0}) +   K_{l-1}(\widehat{\varphi}_l)(x_{0})- K_l(\widehat{\varphi}_l)(x_{0}))^2
\Big).
$$
Then by Lemma \ref{lem:poisson_eq_cont} and (A\ref{ass:2}) 2.~it follows that
$$
(K_{l-1}(\widehat{\varphi}_{l-1}-\widehat{\varphi}_l))(x_{0}) +   K_{l-1}(\widehat{\varphi}_l)(x_{0})- K_l(\widehat{\varphi}_l)(x_{0}))^2 \leq
C \|\varphi\|^2 h_l^{2\beta}.
$$
Hence we can conclude that
\begin{equation}\label{eq:main_theo3}
\mathbb{E}[((R_l(N_l) - R_{l-1}(N_l))^2] \leq \frac{1}{N_l^2} C(\|\varphi\|\vee \|\varphi\|_{\textrm{Lip}})^2 h_l^{\beta}.
\end{equation}
The proof is completed by combining \eqref{eq:main_theo2} and \eqref{eq:main_theo3} with \eqref{eq:main_theo1}.
\end{proof}

\section{Verifying the Assumptions}\label{app:verify}

As in the previous appendix, throughout $C$ is a finite constant which may change from line-to-line,
but will not depend on any important parameters (e.g.~$l$, $n$) unless stated.

\begin{proof}[Proof of Proposition \ref{prop:verify}]
We establish (A\ref{ass:1}-\ref{ass:5}) sequentially.

\noindent\textbf{(A\ref{ass:1})}. We note that for any fixed $l$, $i$, via (B\ref{ass:verify}) 1.
$$
\pi_l(x_i|x_{-i}) = \frac{\pi(x_{1:k}}{\int_{\mathsf{X}_i}\pi_l(x_{1:k})\nu_i(dx_i)} \geq \frac{\underline{C}}{\overline{C}\int_{\mathsf{X}_i}\nu_i(dx)}
$$
where $x_{-i}=(x_{1:i-1},x_{i+1:k})$.
So clearly
$$
K_l(x_{1:k},dx_{1:k}') \geq \Big(\frac{\underline{C}}{\overline{C}}\Big)^k \prod_{i=1}^k\frac{1}{\int_{\mathsf{X}_i}\nu_i(dx)}\nu(dx_i).
$$
Hence $K_l$ is uniformly ergodic and the convergence rate is independent of $l$, so (A\ref{ass:1}) is satisfied.

\noindent\textbf{(A\ref{ass:3})}. We have the decomposition, for any fixed $l$, $x_{1:k},y_{1:k}$, $\varphi\in\mathscr{B}_b(\mathsf{X})\cap\textrm{Lip}(\mathsf{X})$
$$
\int_{\mathsf{X}}\varphi(x_{1:k}')[\prod_{j=1}^k \pi_l(x_j'|x_{1:j-1}',x_{j+1:k})-\prod_{j=1}^k \pi_l(x_j'|x_{1:j-1}',y_{j+1:k})] \nu(dx_{1:k}') = 
$$
$$
\int_{\mathsf{X}}\varphi(x_{1:k}') \Big(\sum_{s=1}^k\Big\{\Big(\prod_{j=1}^{s-1}\pi_l(x_j'|x_{1:j-1}',y_{j+1:k})\Big)\Big(\prod_{j=s}^{k}\pi_l(x_j'|x_{1:j-1}',x_{j+1:k})\Big) - 
$$
\begin{equation}\label{eq:main_decomp_verify_a2}
\Big(\prod_{j=1}^{s}\pi_l(x_j'|x_{1:j-1}',y_{j+1:k})\Big)\Big(\prod_{j=s+1}^{k}\pi_l(x_j'|x_{1:j-1}',x_{j+1:k})\Big)\Big\}\Big)\nu(dx_{1:k}').
\end{equation}
The difference in the summand in the integrand is
$$
\Big(\prod_{j=1}^{s-1}\pi_l(x_j'|x_{1:j-1}',y_{j+1:k})\Big)\Big[\pi_l(x_s'|x_{1:s-1}',x_{s+1:k})-\pi_l(x_s'|x_{1:s-1}',y_{s+1:k})\Big]
\Big(\prod_{j=s+1}^{k}\pi_l(x_j'|x_{1:j-1}',x_{j+1:k})\Big).
$$
Then we note
$$
\Big[\pi_l(x_s'|x_{1:s-1}',x_{s+1:k})-\pi_l(x_s'|x_{1:s-1}',y_{s+1:k})\Big] = 
$$
$$
\frac{\pi_l(x_{1:s}',x_{s+1:k})-\pi_l(x_{1:s}',y_{s+1:k})}{\pi_l(x_{1:s-1}',x_{s+1:k})} +
\pi_l(x_{1:s}',y_{s+1:k})\Big[\frac{\pi_l(x_{1:s-1}',y_{s+1:k})-\pi_l(x_{1:s-1}',x_{s+1:k})}{\pi_l(x_{1:s-1}',y_{s+1:k})\pi_l(x_{1:s-1}',x_{s+1:k})}\Big].
$$
Now by using (B\ref{ass:verify}) 1.~and (B\ref{ass:verify}) 2.:
$$
\Big|\frac{\pi_l(x_{1:s}',x_{s+1:k})-\pi_l(x_{1:s}',y_{s+1:k})}{\pi_l(x_{1:s-1}',x_{s+1:k})}\Big| \leq \frac{C}{\underline{C}\int_{\mathsf{X}_s}\nu_s(dx)}|x_{s+1:k}-y_{s+1:k}| \leq C |x_{1:k}-y_{1:k}|.
$$
By using (B\ref{ass:verify}) 1.:
$$
\Big|
\pi_l(x_{1:s}',y_{s+1:k})\Big[\frac{\pi_l(x_{1:s-1}',y_{s+1:k})-\pi_l(x_{1:s-1}',x_{s+1:k})}{\pi_l(x_{1:s-1}',y_{s+1:k})\pi_l(x_{1:s-1}',x_{s+1:k})}\Big]
\Big| \leq 
$$
$$
\frac{\overline{C}}{\underline{C}^2\int_{\mathsf{X}_s}\nu_s(dx)}|\int_{\mathsf{X}_s}[\pi_l(x_{1:s}',y_{s+1:k})-\pi_l(x_{1:s}',x_{s+1:k})\nu_s(dx_s')]|.
$$
Then by using (B\ref{ass:verify}) 2.~it easily follows that 
$$
\Big|
\pi_l(x_{1:s}',y_{s+1:k})\Big[\frac{\pi_l(x_{1:s-1}',y_{s+1:k})-\pi_l(x_{1:s-1}',x_{s+1:k})}{\pi_l(x_{1:s-1}',y_{s+1:k})\pi_l(x_{1:s-1}',x_{s+1:k})}\Big]
\Big| \leq 
C |x_{1:k}-y_{1:k}|.
$$
So in summary, we have established that
$$
\Big|\pi_l(x_s'|x_{1:s-1}',x_{s+1:k})-\pi_l(x_s'|x_{1:s-1}',y_{s+1:k})\Big|\leq C |x_{1:k}-y_{1:k}|
$$
and hence using (B\ref{ass:verify}) 1. we can show that
$$
\Big|\Big(\prod_{j=1}^{s-1}\pi_l(x_j'|x_{1:j-1}',y_{j+1:k})\Big)\Big[\pi_l(x_s'|x_{1:s-1}',x_{s+1:k})-\pi_l(x_s'|x_{1:s-1}',y_{s+1:k})\Big]\times
$$
$$
\Big(\prod_{j=s+1}^{k}\pi_l(x_j'|x_{1:j-1}',x_{j+1:k})\Big)\Big|\leq C |x_{1:k}-y_{1:k}|.
$$
Recalling \eqref{eq:main_decomp_verify_a2} it easily follows that
$$
\Big|\int_{\mathsf{X}}\varphi(x_{1:k}')[\prod_{j=1}^k \pi_l(x_j'|x_{1:j-1}',x_{j+1:k})-\prod_{j=1}^k \pi_l(x_j'|x_{1:j-1}',y_{j+1:k})] \nu(dx_{1:k}')\Big| \leq 
$$
$$
C (\|\varphi\|\vee \|\varphi\|_{\textrm{Lip}})|x_{1:k}-y_{1:k}|
$$
which verifies (A\ref{ass:3}).

\noindent\textbf{(A\ref{ass:2})}. For (A\ref{ass:2}) 1.~this follows almost immediately from (B\ref{ass:verify}) 3.~and is omitted. For 
(A\ref{ass:2}) 2.~we note that
for any fixed $l\geq 1$, $x_{1:k}$, $\varphi\in\mathscr{B}_b(\mathsf{X})$
$$
\int_{\mathsf{X}}\varphi(x_{1:k}')[\prod_{j=1}^k \pi_l(x_j'|x_{1:j-1}',x_{j+1:k})-\prod_{j=1}^k \pi_{l-1}(x_j'|x_{1:j-1}',x_{j+1:k})] \nu(dx_{1:k}') = 
$$
$$
\int_{\mathsf{X}}\varphi(x_{1:k}') \Big(\sum_{s=1}^k\Big\{\Big(\prod_{j=1}^{s-1}\pi_{l-1}(x_j'|x_{1:j-1}',x_{j+1:k})\Big)\Big(\prod_{j=s}^{k}\pi_l(x_j'|x_{1:j-1}',x_{j+1:k})\Big) - 
$$
$$
\Big(\prod_{j=1}^{s}\pi_{l-1}(x_j'|x_{1:j-1}',x_{j+1:k})\Big)\Big(\prod_{j=s+1}^{k}\pi_l(x_j'|x_{1:j-1}',x_{j+1:k})\Big)\Big\}\Big)\nu(dx_{1:k}').
$$
Then a similar argument to the proof of (A\ref{ass:3}) can be adopted, except using  (B\ref{ass:verify}) 3.~instead of  (B\ref{ass:verify}) 2.; hence the proof is omitted.

\noindent\textbf{(A\ref{ass:4})}. We note that the update of the $i^{th}$ co-ordinate, given $x_{1:k}$, in our Gibbs sampler can be written as (omitting the argument $l$ in $X$)
$$
\overline{X}_i' = T_{l,i}([T_{l,1:i-1}^x,x_{i+1:k}],U_i)
$$
where for simplicity of notation, the arguments of $T_{l,1:i-1}$ are omitted and we use the superscript $x$ to denote conditoning on $x_{1:k}$. Then we have
that for any $l\geq 0$, $x_{1:k},y_{1:k},U_{1:k}$
$$
|\xi_l(x_{1:k},U_{1:k}) - \xi_{l}(y_{1:k},U_{1:k})|^2  = \sum_{i=1}^k |T_{l,i}([T_{l,1:i-1}^x,x_{i+1:k}],U_i)-T_{l,i}([T_{l,1:i-1}^y,y_{i+1:k}],U_i)|^2.
$$
Now applying (B\ref{ass:verify}) 4.~we have
$$
\sum_{i=1}^k |T_{l,i}([T_{l,1:i-1}^x,x_{i+1:k}],U_i)-T_{l,i}([T_{l,1:i-1}^y,y_{i+1:k}],U_i)|^2 \leq
$$
$$
\tau |x_{2:k}-y_{2:k}|^2 + \sum_{i=2}^k |T_{l,i}([T_{l,1:i-1}^x,x_{i+1:k}],U_i)-T_{l,i}([T_{l,1:i-1}^y,y_{i+1:k}],U_i)|^2.
$$
Applying (B\ref{ass:verify}) 4.~again we have
$$
\sum_{i=1}^k |T_{l,i}([T_{l,1:i-1}^x,x_{i+1:k}],U_i)-T_{l,i}([T_{l,1:i-1}^y,y_{i+1:k}],U_i)|^2 \leq
\tau |x_{2:k}-y_{2:k}|^2 + 
$$
$$
\tau(\tau |x_{2:k}-y_{2:k}|^2  + |x_{3:k}-y_{3:k}|^2)
+
\sum_{i=3}^k |T_{l,i}([T_{l,1:i-1}^x,x_{i+1:k}],U_i)-T_{l,i}([T_{l,1:i-1}^y,y_{i+1:k}],U_i)|^2 \leq
$$
$$
2(\tau+\tau^2)|x_{2:k}-y_{2:k}|^2+\sum_{i=3}^k |T_{l,i}([T_{l,1:i-1}^x,x_{i+1:k}],U_i)-T_{l,i}([T_{l,1:i-1}^y,y_{i+1:k}],U_i)|^2.
$$
Recursively applying the same argument yields
$$
\sum_{i=1}^k |T_{l,i}([T_{l,1:i-1}^x,x_{i+1:k}],U_i)-T_{l,i}([T_{l,1:i-1}^y,y_{i+1:k}],U_i)|^2 \leq
(k\sum_{i=1}^k\tau^i) |x_{2:k}-y_{2:k}|^2 \leq 
$$
$$
(k\sum_{i=1}^k\tau^i) |x_{1:k}-y_{1:k}|^2.
$$
As $\tau<1/k$ it easily follows that (A\ref{ass:4}) is verified.

\noindent\textbf{(A\ref{ass:5})}. We have that for any $l\geq 1$, $x_{1:k},U_{1:k}$
$$
|\xi_l(x_{1:k},U_{1:k}) - \xi_{l-1}(x_{1:k},U_{1:k})|^2  = \sum_{i=1}^k |T_{l,i}([T_{l,1:i-1},x_{i+1:k}],U_i)-T_{l-1,i}([T_{l-1,1:i-1},x_{i+1:k}],U_i)|^2.
$$
where we have removed the superscripts from the $T_l,T_{l-1}$. 
Our proof will be via strong induction.

Now consider the first term in the sum, 
$$
|T_{l,1}([x_{2:k}],U_1)-T_{l-1,1}([x_{2:k}],U_1)|^2 \leq C h_l^{\beta}
$$
by (B\ref{ass:verify}) 5.. Now for the second summand
$$
|T_{l,2}([T_{l,1},x_{3:k}],U_2)-T_{l-1,2}([T_{l-1,1},x_{3:k}],U_2)|^2 \leq
$$
$$
|T_{l,2}([T_{l,1},x_{3:k}],U_2)-T_{l-1,2}([T_{l,1},x_{3:k}],U_2)|^2 +
|T_{l-1,2}([T_{l,1},x_{3:k}],U_2) -T_{l-1,2}([T_{l-1,1},x_{3:k}],U_2)|^2
$$
Now, using (B\ref{ass:verify}) 5.~for the first term on the R.H.S.~and
(B\ref{ass:verify}) 4.~for the second term
$$
|T_{l,2}([T_{l,1},x_{3:k}],U_2)-T_{l-1,2}([T_{l-1,1},x_{3:k}],U_2)|^2 \leq
$$
$$
C h_l^{\beta} + \tau|T_{l,1}([x_{2:k}],U_1)-T_{l-1,1}([x_{2:k}],U_1)|^2
$$
and then by (B\ref{ass:verify}) 5.:
$$
|T_{l,2}([T_{l,1},x_{3:k}],U_2)-T_{l-1,2}([T_{l-1,1},x_{3:k}],U_2)|^2 \leq C h_l^{\beta}.
$$
Now let us suppose for $j\in\{1,\dots,i-1\}$
$$
|T_{l,j}([T_{l,1:j-1},x_{j+1:k}],U_j)-T_{l-1,j}([T_{l-1,1:j-1},x_{j+1:k}],U_j)|^2\leq C h_l^{\beta}.
$$
Let us consider the $i^{th}$ term
$$
|T_{l,i}([T_{l,1:i-1},x_{i+1:k}],U_i)-T_{l-1,i}([T_{l-1,1:i-1},x_{i+1:k}],U_i)|^2 \leq
$$
$$
|T_{l,i}([T_{l,1:i-1},x_{i+1:k}],U_i)-T_{l-1,i}([T_{l,1:i-1},x_{i+1:k}],U_i)|^2 +
$$
$$ 
|T_{l-1,i}([T_{l,1:i-1},x_{i+1:k}],U_i)-T_{l-1,i}([T_{l-1,1:i-1},x_{i+1:k}],U_i)|^2.
$$
Again, using (B\ref{ass:verify}) 5.~for the first term on the R.H.S.~and
(B\ref{ass:verify}) 4.~for the second term
$$
|T_{l,i}([T_{l,1:i-1},x_{i+1:k}],U_i)-T_{l-1,i}([T_{l-1,1:i-1},x_{i+1:k}],U_i)|^2 \leq
$$
$$
C h_l^{\beta} + \tau\sum_{j=1}^{i-1}|T_{l,j}([T_{l,1:j-1},x_{j+1:k}],U_j)-T_{l-1,j}([T_{l-1,1:j-1},x_{j+1:k}],U_j)|^2.
$$
Applying the induction hypothesis allows us to conclude that
$$
\sum_{i=1}^k |T_{l,i}([T_{l,1:i-1},x_{i+1:k}],U_i)-T_{l-1,i}([T_{l-1,1:i-1},x_{i+1:k}],U_i)|^2 \leq C h_l^{\beta}
$$
so easily follows that (A\ref{ass:5}) is verified.
\end{proof}

\end{document}